\documentclass[a4paper]{amsart}
\usepackage{graphicx}
\usepackage{amssymb}
\usepackage{amsmath}
\usepackage{amsthm}
\usepackage{amscd}
\usepackage[all,2cell]{xy}

\UseAllTwocells \SilentMatrices
\newtheorem{thm}{Theorem}[section]

\newtheorem{cor}[thm]{Corollary}
\newtheorem{lem}[thm]{Lemma}
\newtheorem{exm}[thm]{Example}

\newtheorem{prop}[thm]{Proposition}
\theoremstyle{definition}

\theoremstyle{remark}
\newtheorem{rem}[thm]{\bf Remark}
\numberwithin{equation}{section}

\begin{document}
\title[The lower extension groups and quotient categories]{The lower extension groups and quotient categories}
\author[Xiaofa Chen, Xiao-Wu Chen] {Xiaofa Chen, Xiao-Wu Chen$^*$}

\thanks{$^*$ The corresponding author}
\date{\today}

\thanks{E-mail: cxf2011$\symbol{64}$mail.ustc.edu.cn, xwchen$\symbol{64}$mail.ustc.edu.cn}

\maketitle

\dedicatory{}%
\commby{}%

\begin{abstract}
For a certain full additive subcategory $\mathcal{X}$ of an additive category $\mathcal{A}$, one defines the lower extension groups in  relative homological algebra. We show that these groups are isomorphic to the suspended Hom groups in the Verdier quotient category of the bounded homotopy category of $\mathcal{A}$ by that of $\mathcal{X}$. Alternatively, these groups are isomorphic to the negative cohomological groups of the Hom complexes in the dg quotient category $\mathcal{A}/\mathcal{X}$, where both $\mathcal{A}$ and $\mathcal{X}$ are viewed as dg categories concentrated in degree zero.
\end{abstract}

\section{Introduction}
Let $\mathcal{A}$ be an additive category, and $\mathcal{X}$ be a full additive subcategory of $\mathcal{A}$. In relative homological algebra, $\mathcal{X}$-resolutions and $\mathcal{X}$-coresolutions play the role of projective resolutions and injective resolutions in classical homological algebra.

The extension groups ${\rm Ext}_{\mathcal{X}, -}^n(-, -)$ are obtained by substituting the $\mathcal{X}$-resolution in the contravariant entry of the Hom bifunctor ${\rm Hom}_\mathcal{A}(-, -)$ and then computing the cohomological groups. Dually, the extension groups ${\rm Ext}_{-, \mathcal{X}}^n(-, -)$ are obtained by substituting the $\mathcal{X}$-coresolution in the covariant entry. These groups might be called the \emph{upper extension groups}.

In general, these groups ${\rm Ext}_{\mathcal{X}, -}^n(-, -)$ and ${\rm Ext}_{-, \mathcal{X}}^n(-, -)$ are not related to each other. We mention that under certain conditions,  the upper extension groups are isomorphic to the suspended Hom groups in relative derived categories \cite{Chen11}.

The \emph{lower extension groups} ${\rm Ext}_{\mathcal{X}, n}(-, -)$ are obtained by substituting the $\mathcal{X}$-resolution in the covariant entry of ${\rm Hom}_\mathcal{A}(-, -)$ and then computing the cohomological groups. They enjoy the balanced property, that is, ${\rm Ext}_{\mathcal{X}, n}(-, -)$ can be  obtained alternatively by substituting the $\mathcal{X}$-coresolution in the contravariant entry of ${\rm Hom}_\mathcal{A}(-, -)$.

The lower extension groups arise in Gorenstein homological algebra \cite{EJ} and the general stabilization theory \cite{Bel}. However, they are less well known than the upper extension groups, probably due to the unusual entries to substitute the $\mathcal{X}$-(co)resolutions.

This work studies the lower extension groups in a completely different perspective.  Denote by $\mathcal{A}/[\mathcal{X}]$ the factor category of $\mathcal{A}$ by those morphisms factoring through $\mathcal{X}$. Denote by $\mathbf{K}^b(\mathcal{A})$ and $\mathbf{K}^b(\mathcal{X})$ the bounded homotopy categories of $\mathcal{A}$ and $\mathcal{X}$, respectively. Then we have the Verdier quotient triangulated category $\mathbf{K}^b(\mathcal{A})/\mathbf{K}^b(\mathcal{X})$.

The following canonical functor
$$\Phi\colon \mathbf{K}^b(\mathcal{A})/\mathbf{K}^b(\mathcal{X})\longrightarrow \mathbf{K}^b(\mathcal{A}/[\mathcal{X}])$$
sends a complex $Z$ in $\mathcal{A}$ to $Z$, viewed as a complex in $\mathcal{A}/[\mathcal{X}]$. An innocent problem is when $\Phi$ is an equivalence. A special case of this problem is implicitly treated in \cite{KV}, where it is used to construct the realization functor of a bounded $t$-structure in an algebraic triangulated category.

The following result is the motivation of this work, which extends the corresponding result contained in \cite[the proof in Subsection 3.2]{KV}.

\vskip 5pt

\noindent {\bf Proposition.}\; \emph{Assume that the lower extension groups ${\rm Ext}_{\mathcal{X}, n}(-, -)$ are defined. Then the canonical functor $\Phi$ is an equivalence if and only if ${\rm Ext}_{\mathcal{X}, n}(-, -)$ vanish for all $n\geq 1$.}

\vskip 5pt

For the proof of the above result, we actually show that the lower extension groups are isomorphic to certain suspended Hom groups in $\mathbf{K}^b(\mathcal{A})/\mathbf{K}^b(\mathcal{X})$; see Theorem \ref{thm:1}. We view the additive categories $\mathcal{A}$ and $\mathcal{X}$ as dg categories concentrated in  degree zero. Then we have the dg quotient category $\mathcal{A}/\mathcal{X}$ in the sense of \cite{Kel99, Dri}. We observe that the lower extension groups are isomorphic to the negative cohomological groups of the Hom complexes in $\mathcal{A}/\mathcal{X}$; see Proposition \ref{prop:dg-qu}. Then we have another proof of Theorem \ref{thm:1} under slightly different assumptions.

To justify the title, we observe that the following three quotient categories are involved: the additive quotient $\mathcal{A}/[\mathcal{X}]$, the triangulated quotient $\mathbf{K}^b(\mathcal{A})/\mathbf{K}^b(\mathcal{X})$ and the dg quotient $\mathcal{A}/\mathcal{X}$. It is well known that they are related as follows: $\mathcal{A}/[\mathcal{X}]$ is equivalent to the homotopy category $H^0(\mathcal{A}/\mathcal{X})$ of $\mathcal{A}/\mathcal{X}$, and  $\mathbf{K}^b(\mathcal{A})/\mathbf{K}^b(\mathcal{X})$ is equivalent to the triangulated hull $(\mathcal{A}/\mathcal{X})^{\rm tr}$ of $\mathcal{A}/\mathcal{X}$.

The paper is organized as follows. In Section 2, we prove that the lower extension groups are isomorphic to the Tor groups of certain modules over $\mathcal{X}$. We prove Theorem \ref{thm:1} in Section 3. The canonical functor $\Phi$ is studied in Section 4. Moreover, a new characterization of a hereditary abelian category is given; see Corollary \ref{cor:here}. In Section 5, we study the dg quotient category $\mathcal{A}/\mathcal{X}$ and prove Proposition \ref{prop:dg-qu}. Then we interpret the Hom groups in the Verdier quotient category $\mathbf{K}^b(\mathcal{A})/\mathbf{K}^b(\mathcal{X})$ as the Tor groups, yielding another proof of Theorem \ref{thm:1}; see Proposition \ref{prop:dg-tor}.

In the sequel, we sometimes abbreviate  ${\rm Hom}_\mathcal{A}(-, -)$ as $\mathcal{A}(-, -)$. We use the cohomological notation for complexes.

\section{The lower extensions as Tor groups}

Let $\mathcal{A}$ be an additive category and $\mathcal{X}\subseteq \mathcal{A}$ a full additive subcategory. Denote by $[\mathcal{X}]$ the two-sided ideal formed by morphisms factoring through $\mathcal{X}$. Then we have the factor category $\mathcal{A}/[\mathcal{X}]$. For two objects $A$ and $B$, we have
$$\mathcal{A}/{[\mathcal{X}]}(A, B)=\mathcal{A}(A, B)/{[\mathcal{X}](A, B)}.$$
The corresponding  coset of a morphism $f\colon A\rightarrow B$ in $\mathcal{A}$ is denoted by $[f]$.

Recall that a \emph{$\mathcal{X}$-resolution} of an object $B$ means a complex in $\mathcal{A}$
$$X_B\colon \cdots \longrightarrow X_B^{-2} \stackrel{d^{-2}}\longrightarrow X^{-1}_B\stackrel{\partial}\longrightarrow B\longrightarrow 0$$
 such that $X_B^{-i}\in \mathcal{X}$ for each $i\geq 1$ and that ${\rm Hom}_\mathcal{A}(X, X_B)$ is acyclic for each object $X\in \mathcal{X}$. In particular, the morphism $\partial$ is a right $\mathcal{X}$-approximation of $B$, that is, any morphism $t\colon T\rightarrow B$ with $T\in \mathcal{X}$ factors through $\partial$. Dually, a \emph{$\mathcal{X}$-coresoultion} of $A$ means a complex
 $$_AX\colon  0\longrightarrow A \longrightarrow {_AX}^1 \longrightarrow {_AX}^2\longrightarrow \cdots $$
 such that each $_AX^i$ lies in $\mathcal{X}$ and that ${\rm Hom}_\mathcal{A}(_AX, X)$ is acyclic for each $X\in \mathcal{X}$.

We assume that the above $\mathcal{X}$-resolution and $\mathcal{X}$-coresolution exist. Then the \emph{upper extension groups} ${\rm Ext}_{\mathcal{X}, -}^n(B, A)=H^{n+1}{\rm Hom}_\mathcal{A}(X_B^{\leq -1}, A)$ and ${\rm Ext}_{-, \mathcal{X}}^n(B, A)=H^{n+1}{\rm Hom}_\mathcal{A}(B, {_AX}^{\geq 1})$ for $n\in \mathbb{Z}$. Here, $X_B^{\leq -1}$ and ${_AX}^{\geq 1}$ denote the brutal truncations of the relevant complexes.

In general, these groups ${\rm Ext}_{\mathcal{X}, -}^n(B, A)$ and ${\rm Ext}_{-, \mathcal{X}}^n(B, A)$ are not related to each other. Therefore, the following balanced property is quite different.

\begin{lem}\label{lem:bala}
We assume that the above $\mathcal{X}$-resolution $X_B$ and $\mathcal{X}$-coresolution $_AX$ exist. Then for each $n\in \mathbb{Z}$, there is an isomorphism
$$H^{-n}{\rm Hom}_\mathcal{A}(A, X_B)\simeq H^{-n}{\rm Hom}_\mathcal{A}({_AX}, B).$$
\end{lem}

\begin{proof}
This follows immediately by considering the collapsing spectral sequences associated to the Hom bicomplex ${\rm Hom}_\mathcal{A}({_AX}, X_B)$.
\end{proof}

The above common cohomology groups are denoted by ${\rm Ext}_{\mathcal{X}, n}(A, B)$, called the \emph{lower extension groups}. We observe that ${\rm Ext}_{\mathcal{X}, -n}(A, B)=0$ for $n\geq 1$ and that
$${\rm Ext}_{\mathcal{X}, 0}(A, B)\simeq \mathcal{A}/[\mathcal{X}](A, B).$$

In what follows, if either $X_B$ or $_AX$ exists, we still talk about the lower extension groups ${\rm Ext}_{\mathcal{X}, n}(A, B)$.

Assume that $\mathcal{X}$ is skeletally small. Denote by $\mathcal{X}\mbox{-Mod}$ the abelian category of left $\mathcal{X}$-modules. Here,  a left $\mathcal{X}$-module is by definition an additive functor $\mathcal{X}\rightarrow \mathbb{Z}\mbox{-Mod}$. Dually, $\mbox{Mod-}\mathcal{X}$ denotes the category of right $\mathcal{X}$-modules. Then we have the well-defined tensor bifunctor
$$-\otimes_\mathcal{X}-\colon \mbox{Mod-}\mathcal{X}\times \mathcal{X}\mbox{-Mod}\longrightarrow \mathbb{Z}\mbox{-Mod}$$
and the corresponding Tor groups ${\rm Tor}^\mathcal{X}_n(-, -)$ for $n\geq 1$.

For example, $\mathcal{A}(-, B)$ and $\mathcal{A}(A, -)$ will be viewed as a right $\mathcal{X}$-module and a left $\mathcal{X}$-module, respectively. Then the tensor product is explicitly given by
$$\mathcal{A}(-, B)\otimes_\mathcal{X} \mathcal{A}(A, -)=(\bigoplus_{X\in \mathcal{X}} \mathcal{A}(X, B)\otimes_\mathbb{Z} \mathcal{A}(A, X))/I,$$
where $I$ is a $\mathbb{Z}$-submodule generated by $b\circ x\otimes a-b\otimes x\circ a$ for all $x\colon X\rightarrow X'$ in $\mathcal{X}$, $a\colon A\rightarrow X$ and $b\colon X'\rightarrow B$. Then there is a canonical map
\begin{align}\label{equ:can}
{\rm can}\colon  \mathcal{A}(-, B)\otimes_\mathcal{X} \mathcal{A}(A, -)\longrightarrow \mathcal{A}(A, B), \quad g\otimes f\mapsto g\circ f,\end{align}
where $g\colon X\rightarrow B$ and $f\colon A\rightarrow X$ for some object $X\in \mathcal{X}$. The following sequence is exact by definition
$$\mathcal{A}(-, B)\otimes_\mathcal{X} \mathcal{A}(A, -)\stackrel{{\rm can}}\longrightarrow \mathcal{A}(A, B)\stackrel{{\rm pr}}\longrightarrow \mathcal{A}/[\mathcal{X}](A, B)\longrightarrow 0,$$
where ``${\rm pr}$" denotes the projection.

\begin{prop}\label{prop:ext-tor}
Assume that $\mathcal{X}$ is skeletally small and that either $X_B$ or $_AX$ exists. Then there are isomorphisms
$${\rm Ext}_{\mathcal{X}, n}(A, B)\simeq {\rm Tor}^\mathcal{X}_{n-1}(\mathcal{A}(-, B), \mathcal{A}(A, -))$$
for $n\geq 2$; moreover, ${\rm Ext}_{\mathcal{X}, 1}(A, B)$ is isomorphic to the kernel of (\ref{equ:can}).
\end{prop}

\begin{proof}
We assume that the $\mathcal{X}$-resolution $X_B$ exists, and the case where $_AX$ exists is similar.

The $\mathcal{X}$-resolution $X_B$ gives rise to the following projective resolution of the right $\mathcal{X}$-module $\mathcal{A}(-, B)$
$$\cdots \longrightarrow \mathcal{X}(-, X_B^{-2}) \stackrel{\mathcal{X}(-, d^{-2})}\longrightarrow \mathcal{X}(-, X_B^{-1}) \longrightarrow \mathcal{A}(-, B)\longrightarrow 0.$$
Applying $-\otimes_\mathcal{X}\mathcal{A}(A, -)$ to it and using the natural isomorphisms
$$\mathcal{X}(-, X_B^{-n})\otimes_\mathcal{X} \mathcal{A}(A, -)\simeq \mathcal{A}(A, X_B^{-n}), $$
the required isomorphisms follow immediately. For the last isomorphism, we just observe that the cokernel of $\mathcal{A}(A, d^{-2})$ is isomorphic to $\mathcal{A}(-, B)\otimes_\mathcal{X} \mathcal{A}(A, -)$.
\end{proof}

\begin{rem}
Since the Tor groups have the balanced property, the above isomorphisms yield another proof of the balanced property of the lower extension groups in Lemma \ref{lem:bala}.
\end{rem}

\begin{exm}\label{exm:here}
{\rm Let $\mathcal{A}$ be an abelian category with enough projectives and enough injectives. Denote by $\mathcal{P}$ (\emph{resp}. $\mathcal{I}$) the full subcategory of projective objects (\emph{resp}. injective objects). The factor category $\mathcal{A}/[\mathcal{P}]$ is usually denoted by $\underline{\mathcal{A}}$. Similarly, we write $\mathcal{A}/[\mathcal{I}]$ as $\overline{\mathcal{A}}$. These factor categories are known as the \emph{stable categories}.

We claim that ${\rm Ext}_{\mathcal{P},n}(-, B)=0$ if and only if ${\rm proj.dim}\; B\leq n$. It suffices to show the ``only if" part. Assume that ${\rm proj.dim}\; B>n$. Then in a projective resolution of $B$
$$\cdots \longrightarrow P^{-n-1}\stackrel{d^{-n-1}} \longrightarrow P^{-n}\longrightarrow \cdots \longrightarrow P^{-1}\longrightarrow B\longrightarrow 0$$
the image $A={\rm Im}\; d^{-n-1}$ is not projective. Then the inclusion $A\rightarrow P^{-n}$ yields a nonzero element in ${\rm Ext}_{\mathcal{P},n}(A, B)$, a contradiction.

Indeed, the above claim can be deduced from the following isomorphism
\begin{align}\label{equ:lower}
{\rm Ext}_{\mathcal{P},n}(A, B)\simeq \underline{\mathcal{A}}(A, \Omega^n(B))
\end{align}
for all $n\geq 0$. Here, $\Omega^n(B)$ denotes the $n$-th syzygy of $B$ and $\Omega^0(B)=B$ by convention.

Recall that an abelian category $\mathcal{A}$ is \emph{hereditary} if its global dimension is at most one. The above claim yields the following characterization of hereditary abelian category: $\mathcal{A}$ is hereditary if and only if ${\rm Ext}_{\mathcal{P},n}(-, -)=0$ for any $n\geq 1$. Dually, $\mathcal{A}$ is hereditary if and only if ${\rm Ext}_{\mathcal{I}, n}(-, -)=0$ for any $n\geq 1$.
}\end{exm}

The following example is studied in \cite[Subsection 3.2]{KV}.

\begin{exm}\label{exm:KV}
{\rm
Let $\mathcal{E}$ be a Frobenius exact category. Denote by $\mathcal{P}$ the full subcategory formed by all the projective-injective objects. The stable category $\underline{\mathcal{E}}=\mathcal{E}/[\mathcal{P}]$ is naturally triangulated, whose suspension functor is denoted by $\Sigma$.

Let $\mathcal{A}\subseteq \mathcal{E}$ be a full additive subcategory containing $\mathcal{P}$.  Denote the factor category $\mathcal{A}/[\mathcal{P}]$ by $\underline{\mathcal{A}}$.  Similar to (\ref{equ:lower}),  we observe  an isomorphism
$${\rm Ext}_{\mathcal{P},n}(A, B)\simeq \underline{\mathcal{E}}(\Sigma^n(A), B)$$
for any $A, B\in \mathcal{A}$ and each $n\geq 0$.}\end{exm}

\section{The lower extensions as suspended Hom groups}

Denote by $\mathbf{K}^b(\mathcal{A})$ be homotopy category of bounded complexes in $\mathcal{A}$. The suspension functor is denoted by $\Sigma$. We will identify an object $A\in \mathcal{A}$ with the corresponding stalk complex concentrated in degree zero. Then $\mathcal{A}$ is viewed as a full subcategory of $\mathbf{K}^b(\mathcal{A})$. For each $n\in \mathbb{Z}$, the suspended stalk complex $\Sigma^{-n}(X)$ is concentrated in degree $n$.

Similarly, we have the homotopy category $\mathbf{K}^b(\mathcal{X})$ of bounded complexes in $\mathcal{X}$. It is a triangulated subcategory of $\mathbf{K}^b(\mathcal{A})$. Denote by $\mathbf{K}^b(\mathcal{A})/\mathbf{K}^b(\mathcal{X})$ the Verdier quotient triangulated category.

We will assume that the category $\mathbf{K}^b(\mathcal{A})/\mathbf{K}^b(\mathcal{X})$  is well defined, that is, all the Hom groups in $\mathbf{K}^b(\mathcal{A})/\mathbf{K}^b(\mathcal{X})$  are sets. For example, this happens provided that  $\mathcal{X}$ is skeletally small.

\begin{thm}\label{thm:1}
Let $A, B$ be two objects in $\mathcal{A}$, which are also viewed as objects in $\mathbf{K}^b(\mathcal{A})/\mathbf{K}^b(\mathcal{X})$. Then the following statements hold.
\begin{enumerate}
\item ${\rm Hom}_{\mathbf{K}^b(\mathcal{A})/\mathbf{K}^b(\mathcal{X})}(\Sigma^{-n}(A), B)=0$ for $n\geq 1$.
\item The natural map ${\mathcal{A}/[\mathcal{X}]}(A, B)\rightarrow {\rm Hom}_{\mathbf{K}^b(\mathcal{A})/\mathbf{K}^b(\mathcal{X})}(A, B)$ is an isomorphism. Consequently, the canonical functor  $\mathcal{A}/[\mathcal{X}]\rightarrow \mathbf{K}^b(\mathcal{A})/\mathbf{K}^b(\mathcal{X})$ is fully faithful.
    \item Assume that either $X_B$ or $_AX$ exists. Then there are isomorphisms
    $${\rm Ext}_{\mathcal{X}, n}(A, B)\simeq {\rm Hom}_{\mathbf{K}^b(\mathcal{A})/\mathbf{K}^b(\mathcal{X})}(\Sigma^n(A), B)$$
    for all $n\geq 1$.
    \end{enumerate}
\end{thm}

\begin{proof}
Recall that any morphism from $\Sigma^m(A)$ to $B$ in $\mathbf{K}^b(\mathcal{A})/\mathbf{K}^b(\mathcal{X})$ is realized as a right roof
\begin{align}\label{equ:rf}
\Sigma^m(A) \stackrel{g}\longrightarrow C \Longleftarrow B
\end{align}
where $C={\rm Cone}(f)$ is the mapping cone of a chain map $f\colon X\rightarrow B$ for some bounded complex $X\in \mathbf{K}^b(\mathcal{X})$. In other words, $C$ has the following form
$$\cdots \longrightarrow X^0\stackrel{\begin{pmatrix}
{f^0}\\
{-d_X^0}
\end{pmatrix}}\longrightarrow B\oplus X^1 \longrightarrow X^2\longrightarrow \cdots $$
and $B\Longrightarrow C$ is the inclusion. Here, by the double arrow, we indicate a morphism which is localized in forming the Verdier quotient. If $m<0$, then $g \colon \Sigma^m(A)\rightarrow C$ factors through the brutal truncation $C^{\geq -m}$. Since $C^{\geq -m}$ lies in $\mathbf{K}^b(\mathcal{X})$, $g$ becomes zero in $\mathbf{K}^b(\mathcal{A})/\mathbf{K}^b(\mathcal{X})$. This proves (1).

For (2), we observe the following commutative diagram
\[\xymatrix{
& C \\
A\ar[dr]_-{a} \ar[ur]^-{g} \ar[r] & C^{\geq 0} \ar[d]_-{\rm pr} \ar[u]^-{\rm inc} & B \ar@{=>}[l] \ar@{=>}[ul] \ar@{=}[dl]\\
& B,
}\]
where ``${\rm inc}$" and ``${\rm pr}$" denote the obvious inclusion and projection, respectively. Here, the map $a$ is obtained by $g^0=\begin{pmatrix} a\\x \end{pmatrix}$ for some morphism $x\in A\rightarrow X^1$. It follows that the given morphism is equivalent to the trivial roof
$$A \longrightarrow B \stackrel{{\rm Id_B}}\Longleftarrow B.$$
In other words, the natural map
 $${\rm Hom}_{\mathcal{A}/[\mathcal{X}]}(A, B)\longrightarrow {\rm Hom}_{\mathbf{K}^b(\mathcal{A})/\mathbf{K}^b(\mathcal{X})}(A, B)$$
 is surjective. For its injectivity, take a morphism $b\colon A\rightarrow B$ which is mapped to zero. It means that $b$ factors through some complex $Y\in \mathbf{K}^b(\mathcal{X})$. Then it follows that $b$ actually factors through $Y^0$ in $\mathcal{A}$, that is, $[b]=0$ in $\mathcal{A}/[\mathcal{X}]$, as required.

 To prove (3), we assume that the $\mathcal{X}$-resolution $X_B$ exists.  The case where $_AX$ exists is similar, where we use left roofs instead of right roofs.

 The required isomorphism
 $$\phi\colon {\rm Ext}_{\mathcal{X}, n}(A, B)\longrightarrow  {\rm Hom}_{\mathbf{K}^b(\mathcal{A})/\mathbf{K}^b(\mathcal{X})}(\Sigma^n(A), B)$$
 sends a class $\bar{c}$ in ${\rm Ext}_{\mathcal{X}, n}(A, B)$ to the right roof
 \begin{align}\label{equ:rf2}
 \Sigma^n(A) \stackrel{c'}\longrightarrow X^{\geq -n}_B \stackrel{\rm inc}\Longleftarrow B.
 \end{align}
 Here, $c\colon A\rightarrow X_B^{-n}$ represents the class $\bar{c}$, $X^{\geq -n}_B $ denotes the brutal truncation of $X_B$,  and the chain map $c'$ is induced by $c$.

 We assume that we are given a right roof (\ref{equ:rf}) with $m=n$. By Lemma \ref{lem:res}(2) below, we infer that the inclusion $B\rightarrow X^{\geq -l}_B$  factor through the natural map $B\rightarrow C={\rm Cone}(f)$ for sufficiently large $l$. Therefore, the right roof (\ref{equ:rf}) is equivalent to
  \begin{align*}
 \Sigma^n(A) \longrightarrow X^{\geq -l}_B \stackrel{\rm inc}\Longleftarrow B.
 \end{align*}
It is clear that the above roof is equivalent to a right roof of the form (\ref{equ:rf2}). This proves the surjectivity of $\phi$.

For the injectivity of $\phi$, we take a morphism $c\colon A\rightarrow X_B^{-n}$ such that the right roof (\ref{equ:rf2}) is equivalent to zero. It follows that $c'\colon \Sigma^n(A)\rightarrow X_B$ factors through some complex $Y\in \mathbf{K}^b(\mathcal{X})$. The factorization is assumed to be $\Sigma^n(A)\stackrel{u} \rightarrow Y\stackrel{v}\rightarrow X_B$. However, by Lemma \ref{lem:res}(1) $v$ is homotopic to zero, and so is $c'$.  In other words, the map $c$ factors through $d^{-n-1}$. Then the corresponding class $\bar{c}$ in ${\rm Ext}_{\mathcal{X}, n}(A, B)$ is zero, as required.
 \end{proof}

Denote by $\mathbf{K}^{-}(\mathcal{A})$ the homotopy category of bounded-above complexes in $\mathcal{A}$.

\begin{lem}\label{lem:res}
Let $X_B$ be the $\mathcal{X}$-resolution of $B$ and $X$ be a bounded complex in $\mathcal{X}$. Then the following statements hold.
\begin{enumerate}
\item ${\rm Hom}_{\mathbf{K}^-(\mathcal{A})}(Y, X_B)=0$ for any complex $Y\in \mathbf{K}^b(\mathcal{X})$.
    \item Given any chain map $f\colon X\rightarrow B$, the inclusion $B\rightarrow X^{\geq -m}_B$ factors through the natural map $B\rightarrow {\rm Cone}(f)$ for sufficiently large $m$.
\end{enumerate}
\end{lem}

\begin{proof}
By definition, ${\rm Hom}_{\mathbf{K}^{-}(\mathcal{A})}(\Sigma^n(Z), X_B)=0$ for any $Z\in \mathcal{X}$ and $n\in \mathbb{Z}$. Then (1) follows immediately.

We apply the cohomological functor ${\rm Hom}_{\mathbf{K}^{-}(\mathcal{A})}(-, X_B)$ to the canonical exact triangle
$$X\stackrel{f} \longrightarrow B \stackrel{\iota}\longrightarrow {\rm Cone}(f) \longrightarrow \Sigma X. $$
By (1) the inclusion $B\rightarrow X_B$ factors through $\iota$. Since ${\rm Cone}(f)$ is a bounded complex, the required factorization follows immediately.
\end{proof}

\section{The canonical functor}

The following canonical functor
$$\Phi\colon \mathbf{K}^b(\mathcal{A})/\mathbf{K}^b(\mathcal{X})\longrightarrow \mathbf{K}^b(\mathcal{A}/[\mathcal{X}])$$
sends a complex $Z$ in $\mathcal{A}$ to $Z$, where the latter is viewed as a complex in $\mathcal{A}/[\mathcal{X}]$.

\begin{prop}\label{prop:can}
Keep the notation as above. Then the following statements hold.
\begin{enumerate}
\item If $\Phi$ is full, then it is dense.
\item The functor $\Phi$ is faithful if and only if it is an equivalence.
\item Assume that the category $\mathcal{A}$ is Krull-Schmidt. Then $\Phi$ is full if and only if it is an equivalence.
\end{enumerate}
\end{prop}

\begin{proof}
(1) Recall that the essential image of a full triangle functor is necessarily a triangulated subcategory. The only triangulated subcategory of $\mathbf{K}^b(\mathcal{A}/[\mathcal{X}])$ containing $\mathcal{A}/[\mathcal{X}]$ is itself. Then the result follows immediately.

(2) Assume that $\Phi$ is faithful. Let $A, B \in \mathcal{A}$. Recall that  ${\rm Hom}_{\mathbf{K}^b(\mathcal{A}/[\mathcal{X}])}(\Sigma^n(A), B)=0$ for $n\neq 0$ and ${\rm Hom}_{\mathbf{K}^b(\mathcal{A}/[\mathcal{X}])}(A, B)\simeq {\mathcal{A}/[\mathcal{X}]}(A, B)$. The faithfulness of $\Phi$ forces that ${\rm Hom}_{\mathbf{K}^b(\mathcal{A})/\mathbf{K}^b(\mathcal{X})}(\Sigma^n(A), B)=0$ for $n\neq 0$. In view of Theorem \ref{thm:1}(2), we have an isomorphism
$${\rm Hom}_{\mathbf{K}^b(\mathcal{A})/\mathbf{K}^b(\mathcal{X})}(\Sigma^n(A), B)\simeq  {\rm Hom}_{\mathbf{K}^b(\mathcal{A}/[\mathcal{X}])}(\Sigma^n(A), B)$$
for each $n\in \mathbb{Z}$. It follows from \cite[Lemma 1]{Bei} that $\Phi$ is fully faithful. By (1), we infer that $\Phi$ is an equivalence.

(3) Recall a well-known fact: a full triangle functor is faithful if and  only if it is faithful on object; see \cite[p.446]{Ric}. Therefore, it suffices to prove that $\Phi$ is faithful on objects, that is, for any complex $Y\in \mathbf{K}^b(\mathcal{A})$, $\Phi(Y)\simeq 0$ implies that $Y$ is isomorphic to a direct summand of some object in $\mathbf{K}^b(\mathcal{X})$.

We observe that $\mathcal{A}/[\mathcal{X}]$ is also Krull-Schmidt. We may assume that $Y$ is a \emph{minimal} complex in $\mathbf{K}^b(\mathcal{A})$, that is, each differential $Y^i\rightarrow Y^{i+1}$ is a radical morphism in $\mathcal{A}$. Then $\Phi(Y)$ is also a minimal complex in $\mathbf{K}^b(\mathcal{A}/[\mathcal{X}])$. Since a null-homotopic minimal complex is necessarily isomorphic to the zero complex in the category of complexes, we infer that each component $Y^i$ is zero in $\mathcal{A}/[\mathcal{X}]$. In other words, each $Y^i$ is isomorphic to a direct summand of some object in $\mathcal{X}$. It follows that $Y$ is isomorphic to a direct summand of some object in $\mathbf{K}^b(\mathcal{X})$, as required.
\end{proof}

\begin{rem}
In general, the denseness of $\Phi$ will not imply its fully-faithfulness; see Example \ref{exm:non-faith}.
\end{rem}

\begin{exm}
{\rm Let $k$ be a field and $\Lambda=k[t]/(t^3)$ be the truncated polynomial algebra. Denote by $\Lambda\mbox{-mod}$ the abelian category of finite dimensional $\Lambda$-modules and by $\Lambda\mbox{-proj}$ the full subcategory formed by projective modules. The stable module category $\Lambda\mbox{-\underline{mod}}$ is by definition $\Lambda\mbox{-mod}/[\Lambda\mbox{-proj}]$. We claim that the canonical functor
$$\Phi\colon \mathbf{K}^b(\Lambda\mbox{-mod})/\mathbf{K}^b(\Lambda\mbox{-proj}) \longrightarrow \mathbf{K}^b(\Lambda\mbox{-\underline{mod}})$$
is not dense. Then it is non-faithful and non-full by Proposition \ref{prop:can}.

Consider the simple $\Lambda$-module $k$ and the $2$-dimensional $\Lambda$-module $M=k[t]/(t^2)$. We have the following complex in $\Lambda\mbox{-\underline{mod}}$
$$0\longrightarrow M \stackrel{[\pi]}\longrightarrow k \stackrel{[\iota]}\longrightarrow M\stackrel{[\pi]} \longrightarrow k \longrightarrow 0,$$
where $\pi$ denotes the projection and $\iota$ is the natural embedding. We observe that this complex does not lie in the essential image of $\Phi$.
}\end{exm}

As pointed out in the introduction, the following result is our main motivation to study the lower extension groups.

\begin{prop}\label{prop:KV}
Assume that $\mathcal{X}\subseteq \mathcal{A}$ satisfies the following condition: $X_B$ exists for each $B\in \mathcal{A}$, or $_AX$ exists for each $A\in \mathcal{A}$. Then the canonical functor $\Phi$ is an equivalence if and only if ${\rm Ext}_{\mathcal{X}, n}(-, -)=0$ for $n\geq 1$.
\end{prop}

\begin{proof}
As we saw in the proof of Proposition \ref{prop:can}(2), the functor $\Phi$ is an equivalence if and only if
$${\rm Hom}_{\mathbf{K}^b(\mathcal{A})/\mathbf{K}^b(\mathcal{X})}(\Sigma^n(A), B)=0$$
for  any objects $A, B\in \mathcal{A}$ and $n\neq 0$. Then we are done by Theorem \ref{thm:1}.
\end{proof}

The ``if" part of the following immediate consequence  is implicitly contained in the proof of \cite[Subsection 3.2]{KV}.

\begin{cor}
Keep the notation as in Example \ref{exm:KV}. Then the canonical functor $ \mathbf{K}^b(\mathcal{A})/\mathbf{K}^b(\mathcal{P})\longrightarrow \mathbf{K}^b(\underline{\mathcal{A}})$ is an equivalence if and only if $\underline{\mathcal{E}}(\Sigma^n(A), B)=0$ for all $A, B\in \mathcal{A}$ and $n\geq 1$. \hfill $\square$
\end{cor}

The following result is a seemingly new characterization of hereditary abelian categories. It follows directly from  Example \ref{exm:here} and Proposition \ref{prop:KV}.

\begin{cor}\label{cor:here}
Let $\mathcal{A}$ be an abelian category with enough projectives and enough injectives. Then the following statements are equivalent:
\begin{enumerate}
\item the canonical functor $\mathbf{K}^b(\mathcal{A})/\mathbf{K}^b(\mathcal{P})\longrightarrow \mathbf{K}^b(\underline{\mathcal{A}})$ is an equivalence;
\item the abelian category $\mathcal{A}$ is hereditary;
    \item the canonical functor $\mathbf{K}^b(\mathcal{A})/\mathbf{K}^b(\mathcal{I})\longrightarrow \mathbf{K}^b(\overline{\mathcal{A}})$ is an equivalence. \hfill $\square$
\end{enumerate}
\end{cor}

\begin{rem}
Assume that the abelian category $\mathcal{A}$ has finite global dimension. Denote by $\mathbf{D}^b(\mathcal{A})$ the bounded derived category of $\mathcal{A}$, and by $\mathbf{K}_{\rm ac}^b(\mathcal{A})$ the homotopy category of bounded acyclic complexes.  Then we have a well-known recollement \cite{BBD}
\[\xymatrix{
\mathbf{K}_{\rm ac}^b(\mathcal{A}) \ar[rr]|{{\rm inc}} &&  \mathbf{K}^b(\mathcal{A}) \ar[rr]|{\rm can}  \ar@/_1pc/[ll] \ar@/^1pc/[ll] &&   \mathbf{D}^b(\mathcal{A}).   \ar@/_1pc/[ll]|{\bf p} \ar@/^1pc/[ll]|{\bf i}
}\]
Here, ${\bf p}$ sends a complex to its projective resolution; in particular, its essential image is $\mathbf{K}^b(\mathcal{P})$. So, we obtain a triangle equivalence
$$\mathbf{K}_{\rm ac}^b(\mathcal{A})\simeq \mathbf{K}^b(\mathcal{A})/\mathbf{K}^b(\mathcal{P}).$$
 Dually, we have the triangle equivalence
 $$\mathbf{K}_{\rm ac}^b(\mathcal{A})\simeq \mathbf{K}^b(\mathcal{A})/\mathbf{K}^b(\mathcal{I}).$$

 Assume now that $\mathcal{A}$ is hereditary. Combining the above equivalences with Corollary \ref{cor:here}, we have a triangle equivalence
 $$ \mathbf{K}^b(\underline{\mathcal{A}})\simeq  \mathbf{K}^b(\overline{\mathcal{A}}).$$
 We mention that if $\mathcal{A}$ is the category of finitely generated modules over an artin algebra, then the stable categories $\underline{\mathcal{A}}$ and $\overline{\mathcal{A}}$ are already equivalent via the Auslander-Reiten translations.
\end{rem}

\begin{exm}\label{exm:non-faith}
{\rm

Let $k$ be a field and $\Lambda=k[t]/(t^2)$ be the algebra of dual numbers. Then the stable module category $\Lambda\mbox{-\underline{mod}}$ is equivalent to the category of finite dimensional $k$-modules. It follows that any complex in $\mathbf{K}^b(\Lambda\mbox{-\underline{mod}})$ is isomorphic to a finite direct sum of stalk complexes. Then the canonical functor
$$\Phi\colon \mathbf{K}^b(\Lambda\mbox{-mod})/\mathbf{K}^b(\Lambda\mbox{-proj}) \longrightarrow \mathbf{K}^b(\Lambda\mbox{-\underline{mod}})$$
is dense. However, since $\Lambda\mbox{-mod}$ is not hereditary, by Corollary \ref{cor:here} the functor $\Phi$ is not an equivalence.
}\end{exm}

\section{The Tor groups and dg quotients}

Let $k$ be a commutative ring, and let $\mathcal{A}$ be a $k$-linear additive category. Assume that $\mathcal{A}$ is skeletally small. By choosing a skeleton, we might assume further that $\mathcal{A}$ is small.

 We view $\mathcal{A}$ and $\mathcal{X}$ as dg categories concentrated in degree zero. The following treatment is similar to \cite[Subsection 7.2]{KY} and \cite[Subsection 5.2]{CC}. For dg quotient categories, we refer to \cite{Kel99, Dri}.

We recall the construction of the dg quotient category. Take a semi-free resolution $\pi\colon \tilde{\mathcal{A}}\rightarrow \mathcal{A}$ as in \cite[Lemma B.5]{Dri}. We identify the objects of $\tilde{\mathcal{A}}$ with those  of $\mathcal{A}$. Denote by $\tilde{\mathcal{X}}$ the full dg subcategory of $\tilde{\mathcal{A}}$ formed by objects in $\mathcal{X}$. Then we have a new dg category $\tilde{\mathcal{A}}/{\tilde{\mathcal{X}}}$ as follows: the objects are the same as in $\tilde{\mathcal{A}}$; for each object $X\in \tilde{\mathcal{X}}$, we freely add a new endomorphism $\varepsilon_X$ of degree $-1$ and set $d(\varepsilon_X)={\rm Id}_X$. For details, we refer to \cite[Subsection 3.1]{Dri}.

By abuse of notation, the resulting dg category $\tilde{\mathcal{A}}/{\tilde{\mathcal{X}}}$ will be denoted by $\mathcal{A}/\mathcal{X}$, called the \emph{dg quotient category} of $\mathcal{A}$ by $\mathcal{X}$. This notation is justified by the fact that $\tilde{\mathcal{A}}/{\tilde{\mathcal{X}}}$ is uniquely determined up to quasi-equivalence; see \cite[1.6.2 Main Theorem]{Dri}.

Thanks to the isomorphisms in Proposition \ref{prop:ext-tor}, the following observation interprets
 the lower extension groups as the negative cohomological groups of the Hom complexes in $\mathcal{A}/\mathcal{X}$.

\begin{prop}\label{prop:dg-qu}
Keep the assumptions and  notation as above. Then for any objects $A, B\in \mathcal{A}/\mathcal{X}$, the Hom complex $\mathcal{A}/\mathcal{X}(A, B)$ is non-positively graded such that the following statements hold.
\begin{enumerate}
\item $H^0(\mathcal{A}/\mathcal{X}(A, B))\simeq \mathcal{A}/[\mathcal{X}](A, B)$.
\item $H^{-1}(\mathcal{A}/\mathcal{X}(A, B))$ is isomorphic to the kernel of  (\ref{equ:can}).
\item For each $n\geq 2$, $H^{-n}(\mathcal{A}/\mathcal{X}(A, B))\simeq {\rm Tor}^\mathcal{X}_{n-1}(\mathcal{A}(-, B), \mathcal{A}(A, -)).$
\end{enumerate}
\end{prop}

\begin{proof}
Recall that the Hom complexes in $\tilde{\mathcal{A}}$ are non-positively graded. By the very construction, the same holds for $\tilde{\mathcal{A}}/{\tilde{\mathcal{X}}}=\mathcal{A}/\mathcal{X}$.

The quasi-equivalence $\pi\colon \tilde{\mathcal{X}}\rightarrow \mathcal{X}$ implies that $\tilde{\mathcal{X}}$ and $\mathcal{X}$ are derived equivalent. The relevant derived equivalences identify the right dg $\tilde{\mathcal{X}}$-module $\tilde{\mathcal{A}}(-, B)$ with the right $\mathcal{X}$-module $\mathcal{A}(-, B)$, and the left dg $\tilde{\mathcal{X}}$-module $\tilde{\mathcal{A}}(A, -)$ with the left $\mathcal{X}$-module $\mathcal{A}(A, -)$. Moreover, the natural map
$$\tilde{\mathcal{A}}(-, B)\otimes^\mathbb{L}_{\tilde{\mathcal{X}}} \tilde{\mathcal{A}}(-, A)\longrightarrow {\mathcal{A}}(-, B)\otimes^\mathbb{L}_{{\mathcal{X}}} {\mathcal{A}}(-, A)$$
is a quasi-isomorphism. Then we have natural isomorphisms
\begin{align}\label{equ:2-tor}
H^{n}(\tilde{\mathcal{A}}(-, B)\otimes^\mathbb{L}_{\tilde{\mathcal{X}}} \tilde{\mathcal{A}}(-, A))\simeq {\rm Tor}^\mathcal{X}_{-n}(\mathcal{A}(-, B), \mathcal{A}(A, -))
\end{align}
for all $n\in \mathbb{Z}$. Here, ${\rm Tor}^\mathcal{X}_0(\mathcal{A}(-, B), \mathcal{A}(A, -))=\mathcal{A}(-, B)\otimes_\mathcal{X} \mathcal{A}(A, -)$ and by convention ${\rm Tor}^\mathcal{X}_n(\mathcal{A}(-, B), \mathcal{A}(A, -))=0$ for $n<0$.

Recall from \cite[Subsection 3.1]{Dri} that there is an exact sequence of complexes
\begin{align}\label{equ:longex}
0\longrightarrow \tilde{\mathcal{A}}(A, B) \longrightarrow \tilde{\mathcal{A}}/{\tilde{\mathcal{X}}}(A, B) \longrightarrow \Sigma ({\bf B}\otimes_{\tilde{\mathcal{X}}} \tilde{\mathcal{A}}(-, A))\longrightarrow 0,
\end{align}
where ${\bf B}$ denotes the total module of the bar resolution \cite[Subsection 6.6]{Kel94} for the right dg $\tilde{\mathcal{X}}$-module $\tilde{\mathcal{A}}(-, B)$, and $\Sigma$ denotes the suspension functor. In particular, by (\ref{equ:2-tor}) we have
$$H^n(\Sigma ({\bf B}\otimes_{\tilde{\mathcal{X}}} \tilde{\mathcal{A}}(-, A)))\simeq {\rm Tor}^\mathcal{X}_{-n-1}(\mathcal{A}(-, B), \mathcal{A}(A, -)).$$
Since $\tilde{\mathcal{A}}(A, B)$ is quasi-isomorphic to the stalk complex $\mathcal{A}(A, B)$ concentrated in degree zero, the long exact sequence associated to (\ref{equ:longex}) yields the required statements.
\end{proof}

For a dg category $\mathcal{C}$, the \emph{homotopy category} $H^0(\mathcal{C})$ is defined such that its objects are the same as $\mathcal{C}$ and its Hom $k$-modules are the zeroth cohomologies  $H^0(\mathcal{C}(A, B))$ of the Hom complexes $\mathcal{C}(A, B)$. Therefore, Proposition \ref{prop:dg-qu}(1) implies that $H^0(\mathcal{A}/\mathcal{X})$ is isomorphic to the factor category $\mathcal{A}/{[\mathcal{X}]}$.

Denote by $\mathcal{C}^{\rm tr}$ the \emph{triangulated hull} of $\mathcal{C}$; for details, see \cite[Subsection 2.4]{Dri}. We mention that  $H^0(\mathcal{C})$   naturally becomes a full subcategory of $\mathcal{C}^{\rm tr}$. Furthermore,  for any objects $A, B$, we have natural isomorphisms
\begin{align}\label{equ:tr}
\mathcal{C}^{\rm tr}(\Sigma^n(A), B) \simeq H^{-n}(\mathcal{C}(A, B))
\end{align}
for all $n\in \mathbb{Z}$. Here, $\Sigma$ denotes the suspension functor on $\mathcal{C}^{\rm tr}$.

In view of  Proposition \ref{prop:ext-tor}, we now actually give another proof of Theorem \ref{thm:1} via the dg method.

\begin{prop}\label{prop:dg-tor}
Let $\mathcal{A}$ be a skeletally small $k$-linear additive category and  $\mathcal{X}\subseteq \mathcal{A}$ be a full additive subcategory. Then there is a triangle equivalence
$$\mathbf{K}^b(\mathcal{A})/\mathbf{K}^b(\mathcal{X})\simeq (\mathcal{A}/\mathcal{X})^{\rm tr}.$$
Consequently, for any $A, B\in \mathcal{A}$, the following isomorphisms hold:
\begin{enumerate}
\item ${\rm Hom}_{\mathbf{K}^b(\mathcal{A})/\mathbf{K}^b(\mathcal{X})}(\Sigma^n(A), B)=0$ for $n<0$;
\item ${\rm Hom}_{\mathbf{K}^b(\mathcal{A})/\mathbf{K}^b(\mathcal{X})}(A, B)\simeq \mathcal{A}/[\mathcal{X}](A, B)$;
    \item ${\rm Hom}_{\mathbf{K}^b(\mathcal{A})/\mathbf{K}^b(\mathcal{X})}(\Sigma^n(A), B)\simeq {\rm Tor}^\mathcal{X}_{n-1}(\mathcal{A}(-, B), \mathcal{A}(A, -))$ for $n\geq 2$; moreover,  ${\rm Hom}_{\mathbf{K}^b(\mathcal{A})/\mathbf{K}^b(\mathcal{X})}(\Sigma(A), B)$ is isomorphic to the kernel of (\ref{equ:can}).
\end{enumerate}
\end{prop}

\begin{proof}
We identify $\mathbf{K}^b(\mathcal{A})$ with $\mathcal{A}^{\rm tr}$, and $\mathbf{K}^b(\mathcal{X})$ with $\mathcal{X}^{\rm tr}$. Then the triangle equivalence follows from a general result \cite[Theorem 3.4 and Subsection 3.5]{Dri}. In view of the isomorphisms (\ref{equ:tr}) for the dg quotient category $\mathcal{A}/\mathcal{X}$,  the remaining statements follow from the triangle equivalence and Proposition \ref{prop:dg-qu}.
\end{proof}

The following consequence is analogous to Proposition \ref{prop:KV}. We omit the same reasoning.

\begin{cor}
Keep the same assumptions as above. Then the canonical functor $\Phi\colon \mathbf{K}^b(\mathcal{A})/\mathbf{K}^b(\mathcal{X})\rightarrow \mathbf{K}^b(\mathcal{A}/[\mathcal{X}])$ is an equivalence if and only if the canonical map (\ref{equ:can}) is injective and ${\rm Tor}^\mathcal{X}_{n}(\mathcal{A}(-, B), \mathcal{A}(A, -))=0$ for all $n\geq 1$. \hfill $\square$
\end{cor}

\vskip 5pt

\noindent{\bf Acknowledgements}.\quad  The authors are grateful to Zhenxing Di and  Yu Ye for helpful comments. This work is supported by the National Natural Science Foundation of China (No.s 11671245 and 11971449),  the Fundamental Research Funds for the Central Universities,  and Anhui Initiative in Quantum Information Technologies (AHY150200).

\bibliography{}

\vskip 10pt

 {\footnotesize \noindent Xiaofa Chen, Xiao-Wu Chen\\
 Key Laboratory of Wu Wen-Tsun Mathematics, Chinese Academy of Sciences,\\
 School of Mathematical Sciences, University of Science and Technology of China, Hefei 230026, Anhui, PR China}

\end{document}